\documentclass[final,a4paper, 11pt]{article}

\usepackage[a4paper,scale=0.768]{geometry}
\usepackage[english]{babel}

\usepackage[mathscr]{euscript}
\usepackage{color}
\usepackage{fancyhdr}
\usepackage{bm}
\usepackage{hyperref}
\usepackage{cite}
\usepackage{showkeys}
\usepackage{subfigure}
\usepackage{float}

\usepackage{amsfonts}
\usepackage{amsmath}
\usepackage{amssymb}
\usepackage{amscd}
\usepackage{amsthm}

\usepackage{accents}
\usepackage{graphicx}
\usepackage{array}

\newtheorem{theorem}{Theorem}

\newtheorem{lemma}[theorem]{Lemma}

\newcommand*{\N}{\ensuremath{\mathbb{N}}}
\newcommand*{\Z}{\ensuremath{\mathbb{Z}}}
\newcommand*{\R}{\ensuremath{\mathbb{R}}}
\newcommand*{\C}{\ensuremath{\mathbb{C}}}
\renewcommand{\i}{\mathrm{i}}
\renewcommand{\epsilon}{{\varepsilon}}

\renewcommand{\d}[1]{\,\mathrm{d}#1 \,}

\newcommand{\F}{\mathcal{F}} 
\newcommand{\J}{\mathcal{J}} 

\newcommand{\E}{{\mathcal{E}}}
\newcommand{\p}{{\mathrm{per}}}

\renewcommand{\O}{{\mathcal{O}}}
\newcommand{\X}{{\mathcal{X}}}

\newlength{\dhatheight}

\usepackage{color}

\begin{document}

\sloppy

\title{Fast convergent PML method for scattering with periodic surfaces: the exceptional case}
\author{
Ruming Zhang\thanks{Institute of mathematics, Technische Universit{\"a}t Berlin, Berlin, Germany; \texttt{zhang@math.tu-berlin.de}}}
\date{}

\maketitle

\begin{abstract}
    In the author's previous paper \cite{Zhang2021b}, exponential convergence was proved for the perfectly matched layers (PML) approximation of scattering problems with periodic surfaces in 2D. However, due to the overlapping of singularities, an exceptional case, i.e., when the wave number is a half integer, has to be excluded in the proof. However, numerical results for these cases still have fast convergence rate and this motivates us to go deeper into these cases. In this paper, we focus on these cases and prove that the fast convergence result for the discretized form. Numerical examples are also presented to support our theoretical results.
\end{abstract}

\section{Introduction}

We focus on the scattering problem with a periodic surface in a two dimensional space. Let $\Gamma$ be a periodic surface in $\R^2$ which is defined as the graph of a periodic and bounded function $\zeta$. $\Omega$ is the half space above $\Gamma$. Let $\Gamma_H:=\R\times\{H\}$ be a straight line lying above $\Gamma$. 
For the visulization we refer to Figure \ref{fig:sample}.
\begin{figure}[h]
    \centering
    \includegraphics[width=0.7\textwidth]{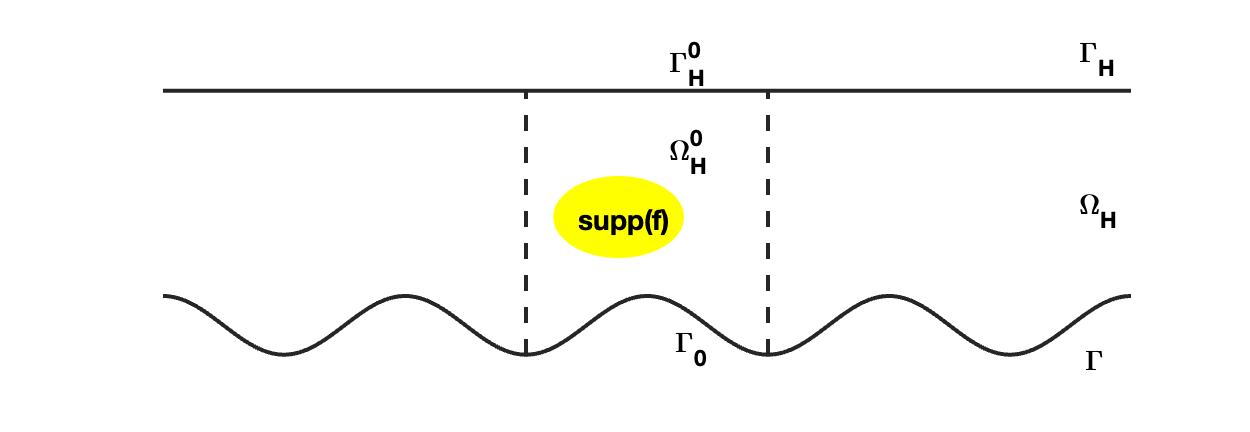}
    \caption{Periodic structures: domains and notations.}
    \label{fig:sample}
\end{figure}

The scattering problem is described by the following equations with a compact supported source term $f$:
\begin{equation}
 \label{eq:org}
 \Delta u+k^2 u=f\text{ in }\Omega;\quad u=0\text{ on }\Gamma
\end{equation}
with the upward propagating radiation condition (UPRC):
\[
u(x)=2\int_{\Gamma}\frac{\partial\Phi(x,y)}{\partial x_2}u(y)\d s(y),\quad x_2>H.
\]
As is proved in \cite{Chand2005}, the UPRC is equivalent to the following transparent boundary condition:
\begin{equation}
 \label{eq:tbc}
 \frac{\partial u}{\partial x_2}=T^+ u=\i\int_\R\sqrt{k^2-\xi^2}\widehat{u}(\xi,H)e^{\i x\xi}\d\xi.
\end{equation}
In \cite{Chand2010}, it is proved that the problem is uniquely solvable in the weighted Sobolev space $H^1_r(\Omega_H)$ for $|r|<1$, where $H^1_r(\Omega_H)$ is defined by the norm
\[
\|\phi\|_{H^1_r(\Omega_H)}=\|(1+x_1^2)^r\phi(x)\|_{H^1_r(\Omega_H)}.
\]

When the source term is quasi-periodic, there is a well-established framework to study these problems and both integral equations (see \cite{Kirsc1993,Arens2010,Zhang2014a}) and finite element methods (\cite{Bao1995,Bao1996,Bao2005}) are applied to produce good numerical approximations. However, for non-periodic incident fields, this framework no longer works thus new methods are necessary. If we simply ignore the periodicity of the surface, then techniques for rough surfaces can be applied, see \cite{Meier2000,Arens2003} for integral equation methods, and 
 \cite{Chand2010} for finite section methods. 
For periodic domains, a well-known tool, the Floquet-Bloch transform, is widely used to rewrite the original problem defined in an unbounded periodic domain into an equivalent new problem defined in a bounded domain in a higher dimensional space. A series of numerical methods have been proposed based on this transform, see \cite{Lechl2015e,Lechl2016,Lechl2016a,Lechl2016b,Lechl2017,Zhang2017e}.

Except for the difficulty brought by the unbounded periodic domain, the proposed methods also suffer from the complexity of the non-local Dirichlet-to-Neumann map $T^+$ in \eqref{eq:tbc}. To this end, the perfectly matched layers are adopted to avoid this difficult. We add an absorbing layer above the physical domain, then the propagating wave is absorbed and decays very fast within this layer. Then on the outer boundary of the layer, the Dirichlet or Neumann data is approximated by. In this case, since the boundary condition becomes local and standard, it is very convenient to be implemented by the finite element methods. Since the PML only provides approximated solutions, the convergence result is of essential importance. For periodic domains we  refer to \cite{Chen2003} and for rough surfaces we refer to \cite{Chand2009} for the convergence analysis. Although only linear convergence was proved globally for rough surfaces, the authors made a conjecture that actually exponential convergence holds locally. In \cite{Zhang2021a}, the author proved the exponential convergence for periodic surfaces except for the cases when the wave numbers are half integers, although numerical results show that fast convergence also happens in these cases. This interesting phenomenon motivates us to go deeper into this topic.

Following \cite{Zhang2021b}, we still apply the Floquet-Bloch transform to both the original and the PML problem, and then the solutions are written as the integral of a family of periodic problems, with respect to the Floquet parameter. With the help of the Gauss quadrature rule, the integral is discretized and convergence analysis is carried out on each node. Then finally we get the fast convergence of the discretized form. 

The rest of the paper is organized as follows. In the second section, the Floquet-Bloch transform is applied. Then in the third section, we introduce the PML approximation. Convergence analysis is discussed in Section 4, and numerical examples are presented in the last section.

\section{Application of the Floquet-Bloch transform}

Since the problem is formulated in the periodic strip $\Omega_H$, it is convenient to define the following  
 domains in one periodicity cell. Let $\Omega_0:=\Omega\cap(-\pi,\pi)\times\R$, $\Gamma_0:=\Gamma\cap(-\pi,\pi)\times\R$ and $\Gamma_H^0:=\Gamma_H\cap(-\pi,\pi)\times\R$. For the visulization of the domains we refer to Figure \ref{fig:sample}. For simplicity, suppose that $f$ is compactly supported in $\Omega_0$.

With $w(\alpha,x):=\F u$ where $\F$ is the Floquet-Bloch transform (see Appendix), then $w(\alpha,\cdot)$ is periodic w.r.t. $x_1$ and satisfies the following equations
\begin{equation}
 \label{eq:per}
 \Delta w(\alpha,\cdot)+2\i\alpha\frac{\partial w(\alpha,\cdot)}{\partial x_1}+(k^2-\alpha^2)w(\alpha,\cdot)=e^{-\i\alpha x_1}f\text{ in }\Omega_0;\quad w(\alpha,\cdot)=0\text{ on }\Gamma_0
\end{equation}
with the transparent boundary condition:
\begin{equation}
 \label{eq:tbc_per}
\frac{\partial w(\alpha,\cdot)}{\partial x_2}=T^+_\alpha w(\alpha,\cdot)=\i \sum_{j\in\Z}\sqrt{k^2-(\alpha+j)^2}\widehat{w}(\alpha,j)e^{\i j x_1}\text{ on }\Gamma_H^0,
\end{equation}
where $\widehat{w}(\alpha,j)$ is the $j$-th Fourier coefficient of $w(\alpha,\cdot)\big|_{\Gamma_H^0}$. 
The variational form for \eqref{eq:per}-\eqref{eq:tbc_per} is, find $w(\alpha,\cdot)\in H^1_\p(\Omega_H^0)$ such that
\begin{equation} \label{eq:per_var}
\begin{split}
   \int_{\Omega_H^0}&\left[\nabla w(\alpha,\cdot)\cdot\nabla\overline{\phi}-2\i\alpha\frac{\partial w(\alpha,\cdot)}{\partial x_1}\overline{\phi}-(k^2-\alpha^2)w(\alpha,\cdot)\overline{\phi}\right]\d x\\
   & -2\pi \i\sum_{j\in\Z}\sqrt{k^2-|\alpha+j|^2}\widehat{w}(\alpha,j)\overline{\widehat{\phi}(j)}=-\int_{\Omega_H^0}e^{-\i\alpha\cdot x}f(x)\overline{\phi}(x)\d x
\end{split}
\end{equation}
holds for any test function $\phi\in H^1_\p(\Omega_H^0)$.
The original solution $u$ is given by the inverse Floquet-Bloch transform as:
\begin{equation}
 \label{eq:ifb}
 u(x_1+2\pi j,x_2)=\int_{-1/2}^{1/2}e^{\i\alpha (x_1+2\pi j)}w(\alpha,x)\d\alpha,\quad x\in\Omega_0.
\end{equation}
Note that from \cite{Kirsc1993}, the problem \eqref{eq:per_var} is always well-posed for any $\alpha\in[-1/2,1/2]$. Thus the inverse Floquet-Bloch representation of $u$ is well defined.

From \cite{Zhang2021b}, there are two critical points in $[-1/2,1/2]$, such that at leaset one of the square roots equals to $0$. Let $\kappa:=\min_{n\in\Z}|k-n|\in [0,0.5]$, then there are two integers $j_-$, $j_+$ such that
\[
 \kappa+j_+=k;\quad-\kappa+j_-=-k,
\]
then $\pm\kappa$ are the two critical points. For any $j\neq j_\pm$, the term $\sqrt{k^2-|\alpha+j|^2}\neq 0$ for all $\alpha\in[-1/2,1/2]$. We define the following operators and functions as follows:
\begin{eqnarray*}  &&\left<A(\alpha)\psi,\phi\right>=\int_{\Omega_H^0}\left[\nabla \psi\cdot\nabla\overline{\phi}-2\i\alpha\frac{\partial \psi}{\partial x_1}\overline{\phi}-(k^2-\alpha^2)\psi\overline{\phi}\right]\d x
    -2\pi \i\sum_{j\neq j_\pm}\sqrt{k^2-|\alpha+j|^2}\psi(j)\overline{\widehat{\phi}(j)},\\
  &&  \left<B_\pm\psi,\phi\right>= 2\pi \i\psi(j_\pm)\overline{\widehat{\phi}(j_\pm)},\\
    &&\left<F(\alpha,\cdot),\phi\right>=-\int_{\Omega_H^0}e^{-\i\alpha x}f(x)\overline{\phi(x)}\d x,
\end{eqnarray*}
where $\psi,\phi\in H^1_\p(\Omega_H^0)$. 
The problem \eqref{eq:per}-\eqref{eq:tbc_per} is written as the following operator equation:
\[
S(\alpha)w(\alpha,\cdot)= \left(A(\alpha)-\sqrt{k^2-(\alpha+j_+)^2}B_+-\sqrt{k^2-(\alpha+j_-)^2}B_-\right)w(\alpha,\cdot)=F(\alpha),
\]
where $A(\alpha)$ and $F(\alpha)$ depend analytically on $\alpha$ near $\pm \kappa$. From \cite{Kirsc1993}, the problem is uniquely solvable in $\alpha\in\R$, and the solution $w$ depends continuously on $\alpha$ and there is a constant $C>0$ such that
\[
 \|w(\alpha,\cdot)\|_{H^1_\p(\Omega_0)}= \|S(\alpha)^{-1}F(\alpha,\cdot)\|_{H^1_\p(\Omega_0)}\leq C\|F(\alpha,\cdot)\|.
\]

Recall that when $2k\notin\N$, exponential convergence of the method has already been proved in \cite{Zhang2021b}, thus we only need to focus on the case that $2k\in\N$ in this paper. When $2k$ is even, then $\kappa=0$; when $2k$ is odd, then $\kappa=0.5$. Thus we choose the interval to be $[-1/2,1/2]$ when $2k$ is even, and $[0,1]$ when $2k$ is odd. For simplicity, we only focus on the case that $k$ is an integer, thus $\kappa=0$ and $j_\pm=\pm k$.

When $|\alpha|$ is close to $0$, the square roots have small absolute values. In this case, $A(\alpha)$ is a small perturbation of $S(\alpha)$ thus is uniformly invertible. 
Then
\[
\left[I-\sqrt{k^2-(\alpha+k)^2}A^{-1}(\alpha)B_+-\sqrt{k^2-(\alpha-k)^2}A^{-1}(\alpha)B_-\right]w(\alpha,\cdot)=A^{-1}(\alpha)F(\alpha,\cdot)
\]
For simplicity, let $\widetilde{B}_\pm(\alpha):=A^{-1}(\alpha)B_\pm$ and $\widetilde{F}(\alpha,\cdot):=A^{-1}(\alpha)F(\alpha,\cdot)$. 
From Neumann series,
\begin{align*}
 w(\alpha,\cdot)=\sum_{m=0}^\infty\sum_{n=0}^\infty\sqrt{k^2-(\alpha+k)^2}^m\sqrt{k^2-(\alpha-k)^2}^n \O\left(\widetilde{B}_+(\alpha),\widetilde{B}_-(\alpha),m,n\right) \widetilde{F}(\alpha,\cdot),
\end{align*}
where $\O\left(\widetilde{B}_+(\alpha),\widetilde{B}_-(\alpha),m,n\right)$ depends analytically on $\alpha$ and is the series of all the repetitions of $m$ times $\widetilde{B}_+(\alpha)$ and $n$ times $\widetilde{B}_-(\alpha)$. The number of terms in this series is ${{m+n}\choose{m}}$. Then
\begin{align*}
w(\alpha,\cdot)=\sum_{m=0}^\infty\sum_{n=0}^\infty\sqrt{k^2-(\alpha+k)^2}^m\sqrt{k^2-(\alpha-k)^2}^n w_{m,n}(\alpha,x)
\end{align*}
where $w_{m,n}(\alpha)=\O\left(\widetilde{B}_+(\alpha),\widetilde{B}_-(\alpha),m,n\right)\widetilde{F}(\alpha,\cdot)$ depends analytically on $\alpha\in(-1,1)$.
Let $\alpha=t^2$ for $\alpha\in[0,1)$, and $\alpha=-t^2$ for $\alpha\in(-1,0]$. Then $t\in[0,1)$ and
\[
w(\alpha,\cdot)=\begin{cases}
\displaystyle
 \sum_{m=0}^\infty\sum_{n=0}^\infty\i^m t^{m+n}\sqrt{2k+t^2}^m\sqrt{2k-t^2}^n w_{m,n}(t^2,x),\quad  \alpha\in[0,1);\\\displaystyle
  \sum_{m=0}^\infty\sum_{n=0}^\infty\i^n t^{m+n}\sqrt{2k-t^2}^m\sqrt{2k+t^2}^n w_{m,n}(-t^2,x),\quad  \alpha\in(-1,0].
\end{cases}
\]
Note that all the functions in this form depend analytically for $t\in(-1,1)$. We also get the original solution by changing the variable:
\begin{equation}
    \label{eq:inv_trans}
    u(x)=2\int_0^{1/\sqrt{2}}t e^{\i t^2 x_1}w(t^2,x)\d t+2\int_0^{1/\sqrt{2}}t e^{-\i t^2 x_1} w(-t^2,x)\d t.
\end{equation}
Moreover, the integrands are extended analytically to a small neighbourhood of  $[0,1/\sqrt{2}]$.

\section{PML approximation}

We add a PML layer above $\Gamma_H$ with thickness $\lambda>0$, described by the complex valued function $s(x_2)$ (see Figure \ref{fig:sample_pml}). Let $s(x_2)=1+\rho\widehat{s}(x_2)$, where $\rho>0$ be a parameter, $\widehat{s}(x_2)$ be a sufficiently smooth function:
\[
 \widehat{s}(x_2)=\begin{cases}
   \X\left(\frac{x_2-H}{\lambda}\right)^m,\quad x_2\in[H,H+\lambda];\\
   0,\quad x_2<H;
 \end{cases}
\]
and $\X$ is fixed with positive real and imaginary parts, $m$ an integer. Let
\[
 \sigma:=\int_H^{H+\lambda}s(x_2)\d x_2=\lambda\left(1+\frac{\rho\X}{m+1}\right).
\]
Then $u_\sigma$ is the approximated solution by the PML, with the equation:
\begin{equation}
 \label{eq:PML}
 \frac{\partial^2 u^\sigma}{\partial x_1^2}+\frac{1}{s(x_2)}\frac{\partial}{\partial x_2}\left(\frac{1}{s(x_2)}\frac{\partial u^\sigma}{\partial x_2}\right)+k^2 u^\sigma=f\text{ in }\Omega_{H+\lambda};\quad u^\sigma=0\text{ on }\Gamma\cup\Gamma_{H+\lambda}.
\end{equation}
Since $s(x_2)=1$ when $x_2<H$, then the problem is rewritten in $\Omega_H$ with a modificed transparent boundary condition:
\begin{eqnarray*}
    &&\Delta u^\sigma_2+k^2 u^\sigma=f\text{ in }\Omega_H;\\
    && u^\sigma=0\text{ on }\Gamma;\\
    &&\frac{\partial u^\sigma}{\partial x_2}=T^\sigma u^\sigma=\i\int_\R\sqrt{k^2-\xi^2}\coth\left(-\i\sqrt{k^2-\xi^2}\sigma\right)\widehat{u}^\sigma(\xi,H)e^{\i x\xi}\d\xi.
\end{eqnarray*}

\begin{figure}[h]
    \centering
    \includegraphics[width=0.7\textwidth]{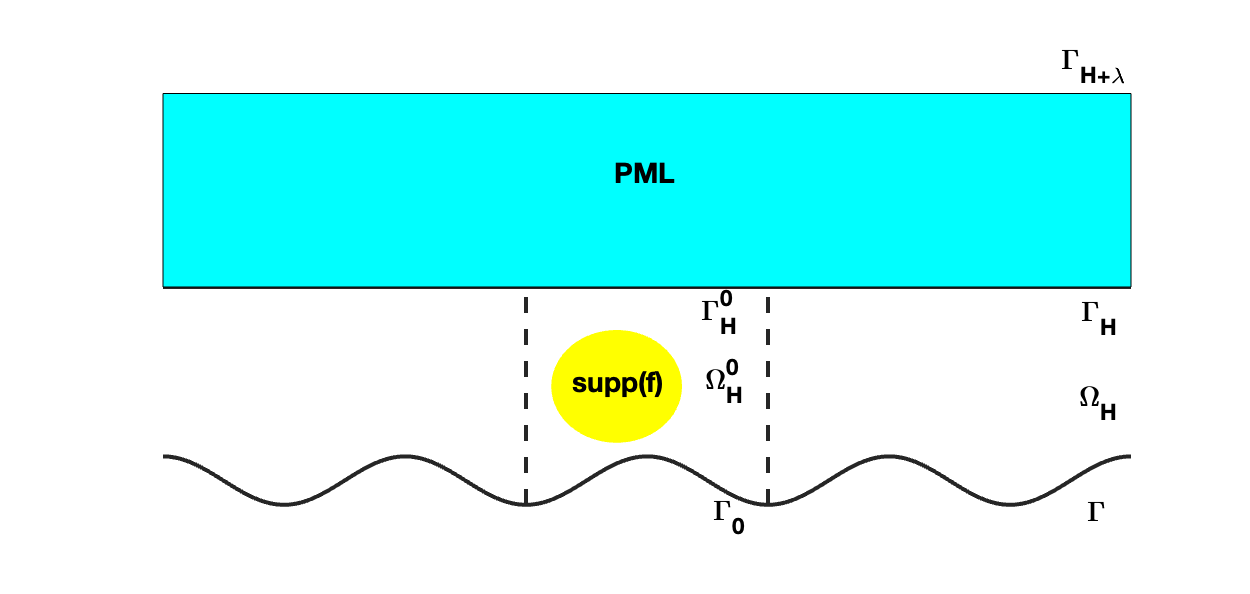}
    \caption{Perfectly matched layers.}
    \label{fig:sample_pml}
\end{figure}

Let $w_\sigma(\alpha,x):=\F u_\sigma$, then from \cite{Zhang2021b}, we get the formulation for the solution:
\begin{equation}
 \label{eq:per_pml}
 \Delta w^\sigma(\alpha,\cdot)+2\i\alpha\frac{\partial w^\sigma(\alpha,\cdot)}{\partial x_1}+(k^2-\alpha^2)w^\sigma(\alpha,\cdot)=e^{-\i\alpha x_1}f\text{ in }\Omega_0;\quad w^\sigma(\alpha,\cdot)=0\text{ on }\Gamma_0
\end{equation}
with the transparent boundary condition:
\begin{equation}
 \label{eq:tbc_per_pml}
\frac{\partial w_\sigma(\alpha,\cdot)}{\partial x_2}=T^\sigma_\alpha w_\sigma(\alpha,\cdot)=\i \sum_{j\in\Z}h(\alpha,\sigma,j)\widehat{w}^\sigma(\alpha,j)e^{\i j x_1}\text{ on }\Gamma_H^0.
\end{equation}
Here $h(\alpha,\sigma,j):=\sqrt{k^2-(\alpha+j)^2}\coth\left(-\i\sqrt{k^2-(\alpha+j)^2}\sigma\right)$.
Similar to the operator $A(\alpha)$, we define $A^\sigma(\alpha)$ by:
\[
\left<A^\sigma(\alpha)\psi,\phi\right>=\int_{\Omega_H^0}\left[\nabla \psi\cdot\nabla\overline{\phi}-2\i\alpha\frac{\partial \psi}{\partial x_1}\overline{\phi}-(k^2-\alpha^2)\psi\overline{\phi}\right]\d x
    -2\pi \i\sum_{j\neq j_\pm}h(\alpha,\sigma,j)\psi(j)\overline{\widehat{\phi}(j)}.
\]
The problem \eqref{eq:per_pml}-\eqref{eq:tbc_per_pml} is written as the following operator equation:
\begin{align*}
S^\sigma(\alpha)w^\sigma(\alpha,\cdot)= &\left(A^\sigma(\alpha)-h(\alpha,\sigma,k)B_+-h(\alpha,\sigma,-k)B_-\right)w_\sigma(\alpha,\cdot)=F(\alpha),
\end{align*}
where $A_\sigma(\alpha)$  depends analytically on $\alpha$ and is also uniformly invertible for $\alpha$ near $0$. Define $\widetilde{B}_\pm^\sigma(\alpha):=A^{-\sigma}(\alpha)B_\pm$ and $\widetilde{F}^\sigma(\alpha,\cdot):=A^{-\sigma}(\alpha)F(\alpha,\cdot)$. 
From Neumann series,
\begin{align*}
 w^\sigma(\alpha,\cdot)&=\sum_{m=0}^\infty\sum_{n=0}^\infty h^m(\alpha,\sigma,k)h^n(\alpha,\sigma,-k) \O\left(\widetilde{B}^\sigma_+(\alpha),\widetilde{B}^\sigma_-(\alpha),m,n\right) \widetilde{F}^\sigma(\alpha,\cdot)\\
 &=\sum_{m=0}^\infty\sum_{n=0}^\infty h^m(\alpha,\sigma,k)h^n(\alpha,\sigma,-k)w^\sigma_{m,n}(\alpha,\cdot).
\end{align*}
Since $h(\alpha,\sigma,\pm k)$ are analytic functions, $w^\sigma(\alpha,\cdot)$ is extended analytically to a neighbourhood of $[-1/2,1/2]$ in the complex plane.

To be align with the integral representation of the original solution $u$, we also change the variable $\alpha$ in the same way. Then $t\in[0,1)$ and
\[
w^\sigma(\alpha,\cdot)=\begin{cases}
\displaystyle
 \sum_{m=0}^\infty\sum_{n=0}^\infty h^m(t^2,\sigma,k)h^n(t^2,\sigma,-k) w_{m,n}^\sigma(t^2,x),\quad  \alpha\in[0,1);\\\displaystyle
  \sum_{m=0}^\infty\sum_{n=0}^\infty h^m(-t^2,\sigma,k)h^n(-t^2,\sigma,-k) w_{m,n}^\sigma(-t^2,x),\quad  \alpha\in(-1,0].
\end{cases}
\]
We also get the solution $u^\sigma$ from the inverse Floquet-Bloch transform by changing the variable:
\begin{equation}
    \label{eq:inv_trans_sigma}
    u^\sigma(x)=2\int_0^{1/\sqrt{2}}t e^{\i t^2 x_1} w^\sigma(t^2,x)\d t+2\int_0^{1/\sqrt{2}}t e^{-\i t^2 x_1} w^\sigma(-t^2,x)\d t.
\end{equation}
Similarly, the two integrands are extended analytically to small neighbourhoods of $[0,1/\sqrt{2}]$.

\section{Discretization and convergence analysis}

\subsection{Gauss-Legendre quadrature rule}
We apply the Gauss-Legendre quadrature rule to discretize the integral on the interval $[0,1/\sqrt{2}]$. Let the integrand  be denoted by $g(\alpha)$ (we omit the variable $x$ here for simplicity), where $g$ is analytic in $[0,1/\sqrt{2}]$. Thus we apply the Gauss-Legendre quadrature rule in each of the interval. Let $\alpha=\frac{t+1}{2\sqrt{2}}$, then
\[
\int_0^{1/\sqrt{2}}g(\alpha)\d\alpha=\frac{1}{2\sqrt{2}}\int_{-1}^{1}g\left(\frac{t+1}{2\sqrt{2}}\right)\d t.
\]
For any positive integer $N$, let $d_j$ be the nodes and $s_j$ be the weights of the Gauss-Legendre quadrature rule, where $j=1,2,\dots,N$. Then
\[
\int_0^{1/\sqrt{2}}g(\alpha)\d\alpha\approx \frac{1}{2\sqrt{2}}\sum_{j=1}^N s_j g\left(\frac{d_j+1}{2\sqrt{2}}\right).
\]
Then $u$ and $u^\sigma$ hav the discretized forms:
\begin{align*}
u_N(x)&=\sum_{j=1}^N s_j\left(\frac{d_j+1}{4}\right) \exp\left(\i x_1\left(\frac{d_j+1}{2\sqrt{2}}\right)^2\right) w\left(\left(\frac{d_j+1}{2\sqrt{2}}\right)^2,x\right)\\
&+\sum_{j=1}^N s_j\left(\frac{d_j+1}{4}\right) \exp\left(-\i x_1\left(\frac{d_j+1}{2\sqrt{2}}\right)^2\right) w\left(-\left(\frac{d_j+1}{2\sqrt{2}}\right)^2,x\right)
\end{align*}
and
\begin{align*}
u_N^\sigma(x)&=\sum_{j=1}^N s_j\left(\frac{d_j+1}{4}\right) \exp\left(\i x_1\left(\frac{d_j+1}{2\sqrt{2}}\right)^2\right) w^\sigma\left(\left(\frac{d_j+1}{2\sqrt{2}}\right)^2,x\right)\\
&+\sum_{j=1}^N s_j\left(\frac{d_j+1}{4}\right) \exp\left(-\i x_1\left(\frac{d_j+1}{2\sqrt{2}}\right)^2\right) w^\sigma\left(-\left(\frac{d_j+1}{2\sqrt{2}}\right)^2,x\right)
\end{align*}
To estimate the error of the quadrature rule, we need the following theorem.

\begin{theorem}[Theorem 5.3.13, \cite{Saute2007}]
  \label{th:err_Gaussian}Let $\E_{-1,1}^\rho\subset\C$ be the ellipse with foci at $(\pm 1,0)$ and sum of the half-axes $\rho>1$.
 Let $\zeta : \, [-1,1] \rightarrow \C$ be real analytic with complex analytic extension to $\E_{-1,1}^\rho$. Denote by $I$ the integral over $(-1,1)$ with integrand $\zeta$ and by $Q_N$ its approximation by the $N$-point Gauss-Legendre quadrature. Then
  \[
    | I - Q_N| \leq C \, \rho^{-2N} \max_{z \in \partial \E^\rho_{-1,1}} |\zeta(z)| \, .
  \]
\end{theorem}

From  analytic extension, the functions $e^{\i\alpha x_1}w(\alpha,x)$ and $e^{\i\alpha x_1}w^\sigma(\alpha,x)$ are extended analytically to  small neighbourhoods of $[-1/2,0]$ and $[0,1/2]$ in the complex plane.  From the change of variable, the integrals are written as the integral on $[-1,1]$, and the complex neighbourhoods are transformed into small neighbourhoods of $[-1,1]$. We can easily find a small ellipse that lies inside the small neighbourhood with $\rho>1$. In this case,  there is a constant $C>0$ such that
\begin{equation}
\label{eq:conv_GL}
\left\|u-u_N\right\|_{H^1(\Omega_H^0)}\leq C\rho^{-2N}.
\end{equation}

\subsection{Convergence of the PML approximation}

 First, we need some basic properties for $s_j$ and $d_j$. From the definition of the weights, we have the following property:
\begin{equation}
    \label{eq:weights}
s_j>0\text{ for }j=1,2,\dots,N;\quad     s_1+s_2+\cdots+s_N=2.
\end{equation}
We also need the estimations for $d_j$. For details we refer to \cite{Bruns1881}.
\begin{theorem}
  \label{th:nodes}
  Let $P_N(\cos\theta)$ be the Legendre polynomial. Let $\theta_1,\theta_2,\dots,\theta_N$ be the zeros of $P_N(\cos\theta)$ in increasing order. Then 
  \[
\frac{2j-1}{2N+1}\pi<\theta_j<\frac{2j}{2N+1}\pi,\quad j=1,2,\dots,N.
  \]
\end{theorem}

To estimate the locations of the nodes $d_j$, we also need the following basic property of the cosine function.

\begin{lemma}
    \label{lm:cos}
    For $\theta\in[0,\pi/2]$, $$1-\frac{\theta^2}{2}\leq
\cos\theta \leq 1-\frac{\theta^2}{3};$$ for $\theta\in[\pi/2,\pi]$, $$-1+\frac{(\pi-\theta)^2}{3}
\leq\cos\theta \leq -1+\frac{(\pi-\theta)^2}{2}.$$
\end{lemma}

\begin{proof}
    For $\theta \in[0,\pi/2]$, from Taylor's series, $\cos\theta\geq 1-\frac{\theta^2}{2}$. Let $r(\theta)=\cos(\theta)+\frac{\theta^2}{3}-1$.
    First, $r(0)=0$. Moreover,
    \[
r'(\theta)=-\sin(\theta)+\frac{2\theta}{3}=-\frac{\theta}{3}+\frac{\theta^3}{3!}-\frac{\theta^5}{5!}\leq 0.
    \]
This implies that $r(\theta)\leq 0$ for $\theta\in[0,\pi/2]$. The second argument comes from the fact that $\cos\theta=-\cos(\pi-\theta)$. 
\end{proof}

Let $d_j=\cos\theta_{N+1-j}$, from the properties of the cosine function, we have:
\[
-1<d_1<d_2<\cdots<d_N<1.
\]
Moreover, from Lemma \ref{lm:cos}, for $j=1,2,\dots,\lfloor\frac{N}{2}\rfloor$ where $\lfloor x\rfloor={\rm floor}(x)$, since $\theta_{N+1-j}\in[\pi/2,\pi]$,
\[
-1+\frac{j^2}{3N^2}<-1+\frac{(2j-1)^2\pi^2}{3(2N+1)^2}\leq d_j\leq -1+\frac{2 j^2\pi^2}{(2N+1)^2}<-1+\frac{5j^2}{N^2}.
\]
Let $\alpha_j=\left(\frac{d_j+1}{2\sqrt{2}}\right)^2$, $\alpha_{j+N}=-\left(\frac{d_j+1}{2\sqrt{2}}\right)^2$, then
\[\alpha_j=\alpha_{j+N}=
\left(\frac{d_j+1}{2\sqrt{2}}\right)^2\geq \frac{j^4}{72N^4}\quad\text{and}\quad \frac{j^2}{3N^2}<d_1+1<\frac{5j^2}{N^2}.
\]
For $j\geq \lfloor\frac{N}{2}\rfloor+1$, since $\theta_{N+1-j}\leq \pi/2$, $d_j\geq 0$. This implies that $\alpha_j\geq 1/8$, $\alpha_{j+N}\leq -1/8$, 

Let $d_{j+N}=d_j$ and $s_{j+N}:=s_j$, where $j=1,2,\dots,N$. 
From the properties of $d_j$
\[
-\frac{1}{2}<\alpha_{2N}<\alpha_{2N-1}<\cdots<\alpha_{N+1}<-\frac{1}{72N^4}<0<\frac{1}{72N^4}<\alpha_{1}<\alpha_{2}<\cdots<\alpha_{N}<\frac{1}{2}.
\] 
Thus $|\alpha_j-n|\geq\frac{1}{72N^4}$ for all integers $n\in\Z$.
For each fixed $(m,n)$, the discretized forms
\[
u_N(x)=\sum_{j=1}^{2N} s_j \left(\frac{d_j+1}{4}\right) e^{\i \alpha_j x_1}w(\alpha_j,x),\quad
\text{and}\quad
u_N^\sigma(x)=\sum_{j=1}^{2N} s_j \left(\frac{d_j+1}{4}\right)e^{\i \alpha_j x_1}w(\alpha_j,x).\]

To compare the above formulas,  we need to estimate the function
\[
\sqrt{k^2-z^2}\left[\coth(-\i\sqrt{k^2-z^2}\sigma)-1\right]=\frac{2\sqrt{k^2-z^2}}{\exp\left[-2\i \sqrt{k^2-z^2}\sigma\right]-1}.
\]
The key result is described in the following lemma.

\begin{lemma}
    \label{lm:est}
    For any $0<\delta<k$ and $z\in (-\infty,-k-\delta]\cup[-k+\delta,k-\delta]\cup[k+\delta,+\infty)$, there are two constants $C_1,C_2>0$ such that
    \[
\left|\frac{2\sqrt{k^2-z^2}}{\exp\left[-2\i \sqrt{k^2-z^2}\sigma\right]-1}\right|\leq C_1|\sigma|^{-1}\exp\left(-C_2\sqrt{k\delta}|\sigma|\right).
\]
\end{lemma}

\begin{proof}
Since $z$ is real, $\sqrt{k^2-z^2}$ is either purely real or purely imaginary. 

\vspace{0.5cm}
\noindent
i) When $|z|\leq k-\delta$, $\sqrt{k^2-z^2}$ is purely real and
\[
\sqrt{k\delta}\leq\sqrt{k^2-z^2}=\sqrt{2k\delta-\delta^2}\leq \sqrt{2k\delta}.
\]
Then
\[
\left|\exp\left[-2\i \sqrt{k^2-z^2}\sigma\right]-1\right|\geq \left|\exp\left[-2\i \sqrt{k^2-z^2}\sigma\right]\right|-1=\exp\left(2\sqrt{k\delta}\sigma_2\right)-1.
\]
Thus
\[
\left|\frac{2\sqrt{k^2-z^2}}{\exp\left[-2\i \sqrt{k^2-z^2}\sigma\right]-1}\right|\leq \frac{2\sqrt{2k\delta}}{\exp\left(2\sqrt{k\delta}\sigma_2\right)-1}=\frac{2\sqrt{2k\delta}}{\left(\exp\left(\sqrt{k\delta}\sigma_2\right)+1\right)\left(\exp\left(\sqrt{k\delta}\sigma_2\right)-1\right)}.
\]
From the mean value theorem, there is a $\xi_1\in[0,\sqrt{k\delta}\sigma_2]$ such that
\[
\exp\left(\sqrt{k\delta}\sigma_2\right)-1=\sqrt{k\delta}\sigma_2\exp\left(\xi_1\right)\geq \sqrt{k\delta}\sigma_2.
\]
Then
\[
\left|\frac{2\sqrt{k^2-z^2}}{\exp\left[-2\i \sqrt{k^2-z^2}\sigma\right]-1}\right|\leq \frac{2\sqrt{2}}{\sigma_2}\exp\left(-\sqrt{k\delta}\sigma_2\right).
\]

\vspace{0.5cm}
\noindent
ii) When $|z|\geq k+\delta$, $\sqrt{k^2-z^2}$ is purely imaginary and
\[
\sqrt{2k\delta}\leq -\i\sqrt{k^2-z^2}=\sqrt{2k\delta+\delta^2}\leq \sqrt{3k\delta}.
\]
Then
\[
\left|\exp\left[-2\i \sqrt{k^2-z^2}\sigma\right]-1\right|\geq \left|\exp\left[-2\i \sqrt{k^2-z^2}\sigma\right]\right|-1=\exp\left(2\sqrt{2k\delta}\sigma_1\right)-1.
\]  
Thus
\[
\left|\frac{2\sqrt{k^2-z^2}}{\exp\left[-2\i \sqrt{k^2-z^2}\sigma\right]-1}\right|\leq \frac{2\sqrt{3k\delta}}{\exp\left(2\sqrt{2k\delta}\sigma_1\right)-1}.
\]
With similar arguments as i), we get
\[
\left|\frac{2\sqrt{k^2-z^2}}{\exp\left[-2\i \sqrt{k^2-z^2}\sigma\right]-1}\right|\leq \frac{\sqrt{6}}{\sigma_1}\exp\left(-\sqrt{2k\delta}\sigma_1\right).
\]
Since $\sigma=\sigma_1+\i\sigma_2$ for positive $\sigma_1,\sigma_2$, we can easily find the constants $C_1$ and $C_2$ to finish the proof.
\end{proof}

From above result, we conclude immediately that
\[
\left|h(\alpha,\sigma,j)-\sqrt{k^2-(\alpha+j)^2}\right|\leq C_1|\sigma|^{-1}\exp\left(-C_2\sqrt{k\delta}|\sigma|\right)
\]
when $|\alpha+j-k|\geq\delta$.
In particular, for any $\alpha_j$ and $\alpha_{j+N}$ ($j=1,2,\dots,\lfloor\frac{N}{2}\rfloor$), since $k$ is a positive integer, $|\alpha_j+\ell|,|\alpha_{j+N}+\ell|\geq \frac{j^4}{72N^4}$ holds uniformly for all $\ell\in\Z$. Thus by modifying the constant $C_2$,
\begin{equation}
    \label{eq:est_diff}
    \left|h(\alpha_j,\sigma,\ell)-\sqrt{k^2-(\alpha_j+\ell)^2}\right|\leq C_1|\sigma|^{-1}\exp\left(-C_2\sqrt{k}|\sigma|(j/N)^2\right).
\end{equation}
For $j=\lfloor\frac{N}{2}\rfloor+1,\dots,N$, since $|\alpha_j-n|,|\alpha_{j+N}-n|\geq 1/8$, there is a constant $C_3$ such that
\begin{equation}
    \label{eq:est_diff_far}
     \left|h(\alpha_j,\sigma,\ell)-\sqrt{k^2-(\alpha_j+\ell)^2}\right|\leq C_1|\sigma|^{-1}\exp\left(-C_3\sqrt{k}|\sigma|\right).
\end{equation}

The following lemma is proved by this result immediately.

\begin{lemma}
 \label{lm:conv_w}
 The function $w^\sigma(\alpha_j,\cdot)$ converges to $w(\alpha_j,\cdot)$ uniformly w.r.t. $j$, 
 \begin{align*}
 & \|w^\sigma(\alpha_j,\cdot)-w(\alpha_j,\cdot)\|_{H^1_\p(\Omega_H^0)},  \|w^\sigma(\alpha_{j+N},\cdot)-w(\alpha_{j+N},\cdot)\|_{H^1_\p(\Omega_H^0)}\\&\leq\begin{cases}
    C_1|\sigma|^{-1}\exp\left(-C_2\sqrt{k}|\sigma|(j/N)^2\right),\quad j=1,2,\dots,\lfloor\frac{N}{2}\rfloor;\\
    C_1|\sigma|^{-1}\exp\left(-C_3\sqrt{k}|\sigma|\right),\quad j=\lfloor\frac{N}{2}\rfloor+1,\dots,N. 
  \end{cases} 
 \end{align*}
 for the fixed constant $C_1,\,C_2,\,C_3>0$.
\end{lemma}

Based on above results, we conclude the following estimation.

\begin{theorem}
    \label{th:est_N}
    For sufficiently large $|\sigma|$ and $N$, the following estimation holds:
    \[
\left\|u_N-u_N^\sigma\right\|<C_5|\sigma|^{-1}N\exp\left(-C_4\sqrt{k}|\sigma|N^{-2}\right)
    \]
    for the fixed constants $C_5,\,C_4>0$.
\end{theorem}

\begin{proof}
    From the discretization of $u$ and $u^\sigma$, using Lemma \ref{lm:conv_w} and \eqref{eq:weights}, \eqref{eq:est_diff}-\eqref{eq:est_diff_far},
    \begin{align*}
\left\|u_N-u_N^\sigma\right\|_{H^1(\Omega_H^0)}\leq &\sum_{j=1}^{2N} s_j\left(\frac{d_j+1}{4}\right)\left\|w(\alpha_j,\cdot)- w^\sigma( \alpha_j ,\cdot)\right\|_{H^1_\p(\Omega_H^0)}\\
\leq&\left[\sum_{j=1}^{\lfloor\frac{N}{2}\rfloor}+\sum_{j=N+1}^{N+\lfloor\frac{N}{2}\rfloor}+\sum_{j=\lfloor\frac{N}{2}\rfloor+1}^N+\sum_{j=\lfloor\frac{N}{2}\rfloor+N+1}^{2N}\right] \\&\quad\cdot s_j\left(\frac{d_j+1}{4}\right)\left\|w(\alpha_j,\cdot)- w^\sigma( \alpha_j ,\cdot)\right\|_{H^1_\p(\Omega_H^0)}\\
\leq &2 \sum_{j=1}^{\lfloor\frac{N}{2}\rfloor} s_j \frac{5j^2}{4N^2}C_1|\sigma|^{-1}\exp\left(-C_2\sqrt{k}|\sigma|(j/N)^2\right)\\&\quad+2\sum_{j=\lfloor\frac{N}{2}\rfloor+1}^N  \frac{s_j}{2} C_1|\sigma|^{-1}\exp\left(-C_3\sqrt{k}|\sigma|\right)\\
\leq & \sum_{j=1}^{\lfloor\frac{N}{2}\rfloor} \frac{5j^2}{N^2}C_1|\sigma|^{-1}\exp\left(-C_2\sqrt{k}|\sigma|N^{-2}\right)+2 C_1|\sigma|^{-1}\exp\left(-C_3\sqrt{k}|\sigma|\right)\\
\leq& C_4 |\sigma|^{-1}N\exp\left(-C_2\sqrt{k}|\sigma|N^{-2}\right)+2 C_1|\sigma|^{-1}\exp\left(-C_3\sqrt{k}|\sigma|\right)
    \end{align*}
The proof is finished by choosing proper constants $C_4$ and $C_5$.
\end{proof}

Together with \eqref{eq:conv_GL}, we get the final conclusion.

\begin{theorem}
    \label{th:final}
 For sufficient large $|\sigma|$ and $N=O\left(\sqrt{|\sigma|}\right)$, there are two constants $c,\,C>0$ such that
 \[
\|u-u^\sigma_N\|\leq C \exp(-c|\sigma|^{1/3}).
 \]
\end{theorem}

\begin{proof}
    From above results, we get the following inequality immediately:
    \[
\|u-u^\sigma_N\|\leq C\rho^{-2N}+C_5|\sigma|^{-1}N\exp\left(-C_4\sqrt{k}|\sigma|N^{-2}\right).
    \]
Note that for a $\rho>1$, there is a constant $C_6>0$ such that $\rho=e^{C_6}$, then $\rho^{-2N}=e^{-2C_6 N}$. Let $N=|\sigma|^{\gamma}$ where $\gamma\in(0,1)$ (if it is not an integer, then it is the closest integer), then
 \[
\|u-u^\sigma_N\|\leq C\exp(-2C_6|\sigma|^{\gamma})+C_5|\sigma|^{-1}\exp\left(-C_4\sqrt{k}|\sigma|^{1-2\gamma}\right).
    \]
Let $\gamma=1/3$, then we have two constants $c,C>0$ such that
\[
\|u-u^\sigma_N\|\leq C\exp\left(-c|\sigma|^{1/3}\right).
\] 
\end{proof}

\section{Numerical examples}

In this section, we present four numerical examples to support our theoretical result. For all the examples, the periodic surface is given by
\[
\Gamma:=\left\{\left(x_1,1.5+\frac{\sin x_1}{3}-\frac{\cos 2x_1}{4}\right):\,x_1\in\R\right\}.
\]
The PML lies in the strip $\R\times[2.5,4]$ and $\mathcal{X}=\exp(\i\pi/4)$. The source term $f$ is compactly supported and defined as
\[f(x_1,x_2)=\begin{cases}
    \cos(2\pi x_1)\sin(2\pi x_2),\quad\text{ when }\sqrt{(x_1+0.4)^2+(x_2-1.8)^2}<0.4;\\
    0,\quad\text{ otherwise.}
\end{cases}
\]
For the visualization of the structures and sources we refer to Figure \ref{fig:ne}.
\begin{figure}[h]
    \centering
    \includegraphics[width=0.45\textwidth]{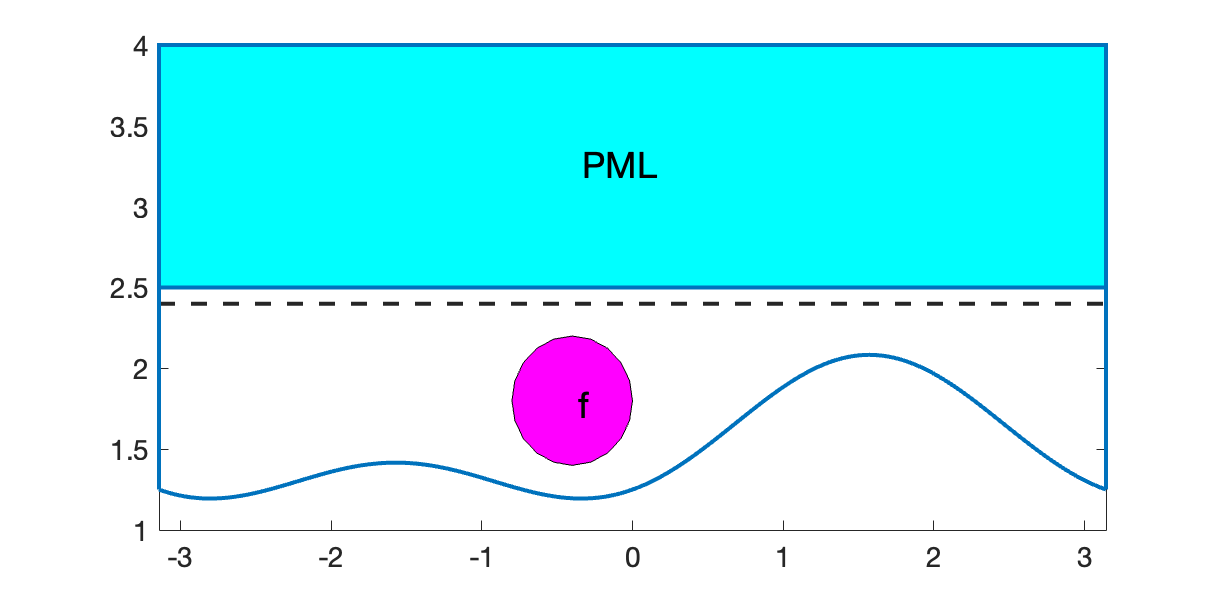}\includegraphics[width=0.45\textwidth]{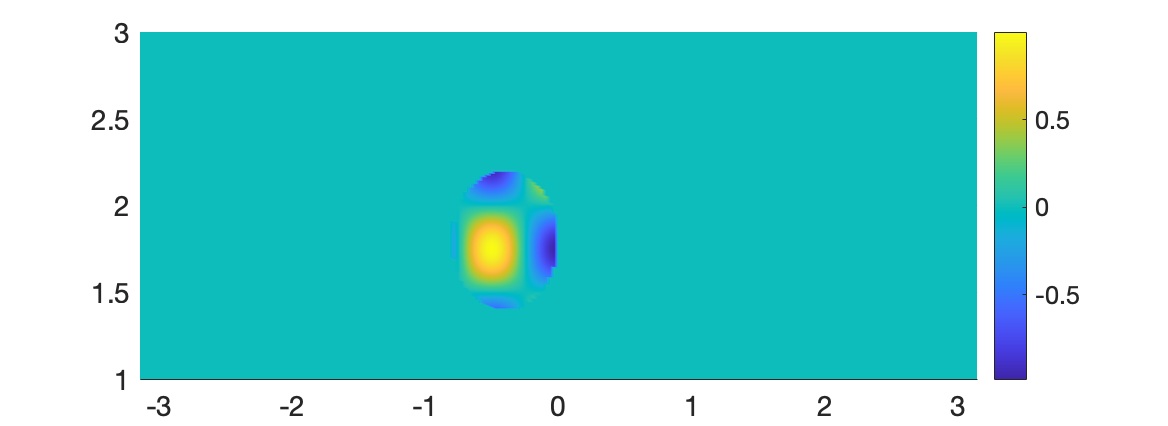}
    \caption{Left: structure; right: source term.}
    \label{fig:ne}
\end{figure}

We take $\rho=2,4,6,8,10,12,14,16$, and let the numerical solution when $\rho=25$ as the reference. To solve the periodic problem \eqref{eq:PML}, we apply the finite element method with the mesh size $0.050$. For each numerical result, we take the value at $[-\pi,\pi]\times\{2.4\}$ and compare the relative $L^2$-norm on this line segment:
\[
{\rm error}_{\rho,k}:=\frac{\|u_{\rho,k}(\cdot,2.4)-u_{25,k}(\cdot,2.4)\|_{L^2(-\pi,\pi)}}{\|u_{\rho,k}(\cdot,2.4)\|_{L^2(-\pi,\pi)}}.
\]
The relative errors for $k=1,1.5,2.5,5$ are listed in Table \ref{table}, and the relative errors are shown in Figure \ref{fig:err}. Note that in the pictures, the $x$-axis is $\log(\rho)$, and the $y$-axis is $\log(-\log({\rm error}_{\rho,k}))$. The red dots are computational results and the the black dashed lines are the linear regressions.

\begin{table}[h]
    \centering
    \begin{tabular}{|c|c|c|c|c|}
    \hline
        $\rho$ & $k=1$ & $k=1.5$ & $k=2.5$ & $k=5$  \\
        \hline
        \hline
       2  & $3.2\times 10^{-1}$ & $1.3\times 10^{-1}$ & $7.6\times 10^{-2}$ & $1.7\times 10^{-2}$ \\
       4  & $5.4\times 10^{-2}$ & $2.4\times 10^{-2}$ & $1.7\times 10^{-2}$ & $8.7\times 10^{-3}$\\
       6  & $1.2\times 10^{-2}$ & $9.4\times 10^{-3}$ & $8.8\times 10^{-3}$ & $1.9\times 10^{-3}$\\
       8  & $3.9\times 10^{-3}$ & $3.0\times 10^{-3}$ & $2.8\times 10^{-3}$ & $1.3\times 10^{-3}$\\
       10 & $2.3\times 10^{-3}$ & $2.3\times 10^{-3}$ & $1.9\times 10^{-3}$ & $6.1\times 10^{-4}$\\
       12 & $4.9\times 10^{-4}$ & $7.9\times 10^{-4}$ & $7.4\times 10^{-4}$ & $3.9\times 10^{-4}$\\
       14 & $5.1\times 10^{-4}$ & $7.3\times 10^{-4}$ & $5.0\times 10^{-4}$ & $2.6\times 10^{-4}$\\
       16 & $1.9\times 10^{-4}$ & $2.4\times 10^{-4}$ & $2.7\times 10^{-4}$ & $1.1\times 10^{-4}$\\
       \hline
    \end{tabular}
    \caption{Relative errors with different $k$'s.}
    \label{table}
\end{table}

\begin{figure}[h]
    \centering
    \includegraphics[width=0.45\textwidth]{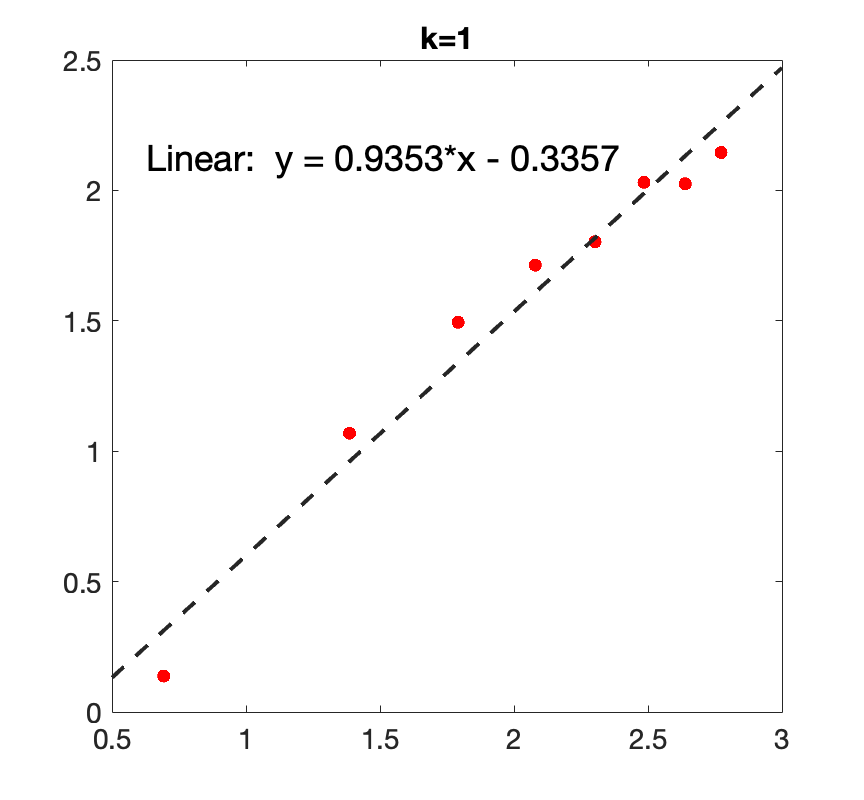}\includegraphics[width=0.45\textwidth]{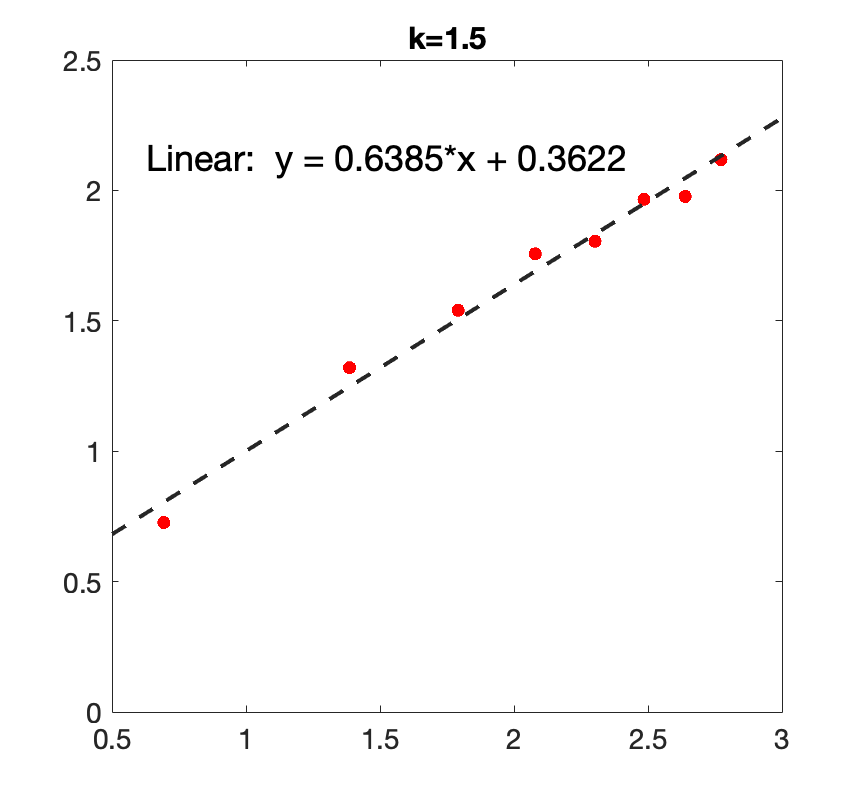}\\
    \includegraphics[width=0.45\textwidth]{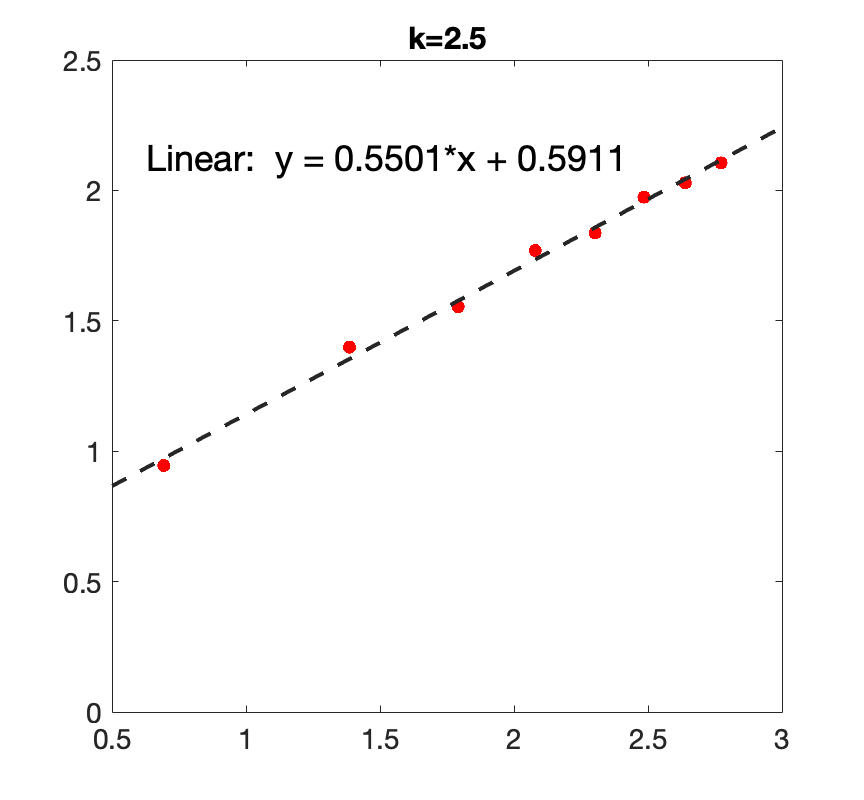}
    \includegraphics[width=0.45\textwidth]{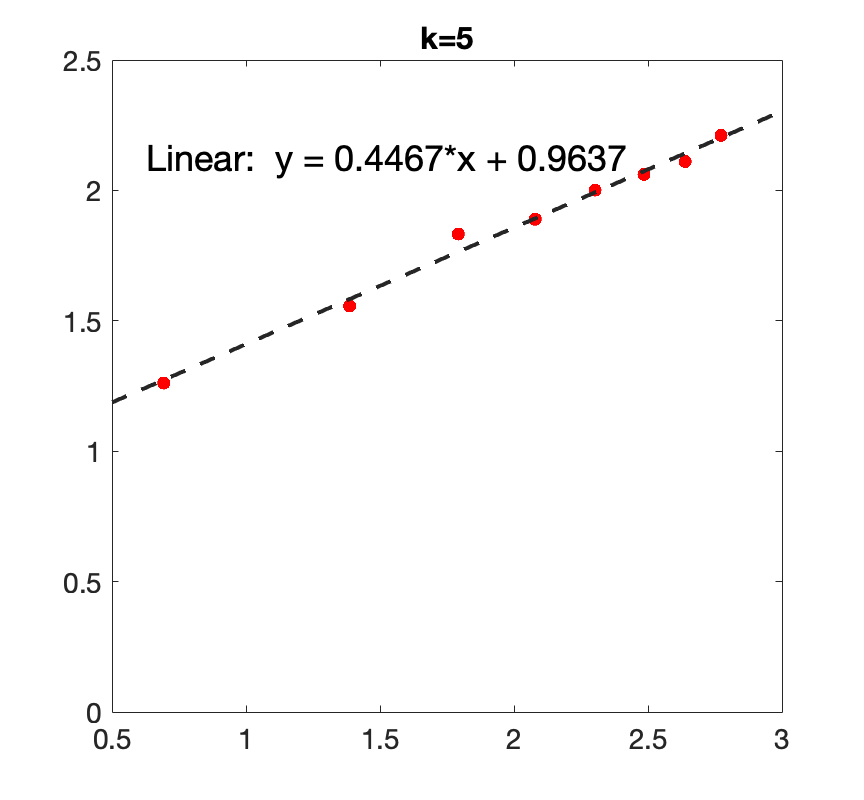}
    \caption{Semi-log plots for relative errors.}
    \label{fig:err}
\end{figure}

From the pictures, the linear regressions fit well for the computational results for all the four wave numbers. This implies that the numerical results coincide with our theoretical approach in Theorem \ref{th:final}. Although the convergence is not exponential, it is still faster than any algebraic order thus it converges super algebraically. This implies that the PML also provides very efficient numerical approximations for the exceptional cases which are excluded in \cite{Zhang2021a}.

\section*{Appendix}

In the appendix, we recall the definition and important properties of the Floquet-Bloch transform. Let $\phi\in C_0^\infty(\Omega_H)$, define the Floquet-Bloch transform by
\[
(\F\phi)(\alpha,x)=\sum_{j\in\Z}\phi(x_1+2\pi j,x_2)e^{-\i\alpha+j(x_1+2\pi j)},\quad x\in\Omega_H^0,\,\alpha\in[-1/2,1/2],
\]
where $\alpha$ is called the Floquet parameter.

It is easily checked that $(\F\phi)(\alpha,\cdot)$ is $2\pi$-periodic in $x_1$ direction, and $e^{\i\alpha x_1}\F\phi)(\alpha,x)$ is $1$-periodic with resepct to $\alpha$.

The definition of $\F$ is also extended to larger function spaces. First we need to define the space $H^m\left((-1/2,1/2];H^s(\Omega_H^0)\right)$ for any non-negative integer $m$ by the norm:
\[
\|\psi\|_{H^m\left((-1/2,1/2];H^s_\p(\Omega_H^{2\pi})\right)}:=\left[\sum_{\ell=0}^m\int_{-1/2}^{1/2} \left\|\partial^\ell_\alpha\psi(\alpha,\cdot)\right\|^2_{H^s_\p(\Omega^{0}_H)}\d\alpha\right]^2.
\]
This definition is extended to all $r>0$ by space interpolation and to negative $r$ from duality arguments. The important properties of the Floquet-Bloch transform is concluded in the following theorem.

\begin{theorem}[Theorem 8, \cite{Lechl2016}]\label{th:Bloch}
 The transform $\F$ is extended to an isomorphism between $H^s_r(\Omega_H)$ and $H^r\left((-1/2,1/2];H^s_\p(\Omega^\Lambda_H)\right)$ for any $s,\,r\in\R$.
 \[\left(\J^{-1} w\right)(x)=\int_{-1/2}^{1/2} w(\alpha,x)e^{\i \alpha x_1}\d\alpha,\quad x\in\Omega_H.
 \] 
\end{theorem}

\section*{Acknowledgements.}
This research was funded by the Deutsche Forschungsgemeinschaft (DFG, German Research Foundation) – Project-ID 258734477 – SFB 1173.

\bibliographystyle{plain}
\providecommand{\noopsort}[1]{}

\end{document}